\documentclass[11pt]{article}
\usepackage{enumerate}
\usepackage{amssymb,a4wide,latexsym,makeidx,epsfig,fleqn}
\usepackage{amsthm}
\usepackage{amsmath}
\usepackage{enumerate}
\usepackage{graphicx}
\usepackage{float}
\usepackage{color}
\allowdisplaybreaks[4]
\newtheorem{theorem}{Theorem}[section]

\newtheorem{definition}[theorem]{Definition}
\newtheorem{lemma}[theorem]{Lemma}

\newtheorem{claim}{Claim}

\begin{document}
\textwidth 150mm \textheight 225mm
\title{Spectral condition for $k$-factor-criticality in $t$-connected graphs\thanks{Supported by the
National Natural Science Foundation of China (No. 12271439) and the China Scholarship Council (No. 202306290194).}}
\author{{Tingyan Ma$^{a,b,c}$, Edwin R. van Dam$^{c}$, Ligong Wang$^{a,b,}$\thanks{Corresponding author.}}\\
{\small $^a$School of Mathematics and Statistics, Northwestern
Polytechnical University,}\\ {\small  Xi'an, Shaanxi 710129,
P.R. China.}\\
{\small $^b$Xi'an-Budapest Joint Research Center for Combinatorics, Northwestern
Polytechnical University,}\\
{\small Xi'an, Shaanxi 710129,
P.R. China,}\\
{\small \small $^c$Department of Econometrics and O.R., Tilburg University, the Netherlands.}\\
{\small E-mail: matingylw@163.com, Edwin.vanDam@uvt.nl, lgwangmath@163.com}}
\date{}
\maketitle
\begin{center}
\begin{minipage}{120mm}
\vskip 0.3cm
\begin{center}
{\small {\bf Abstract}}
\end{center}
{\small A graph $G$ is called $k$-factor-critical if $G-S$ has a perfect matching for every $S\subseteq V(G)$ with $|S|=k$. A connected graph $G$ is called $t$-connected if it has more than $t$ vertices and remains connected whenever fewer than $t$ vertices are removed. We give a condition on the number of edges and a condition on the spectral radius for $k$-factor-criticality in $t$-connected graphs.
}
\end{minipage}
\end{center}
\section{Introduction}
Favaron \cite{Favaron 1996} introduced the notion of $k$-factor-critical graphs.

\begin{definition}\label{de1}
A graph $G$ of order $n$ is said to be $k$-factor-critical if $G-S$ has a perfect matching for every $S\subseteq V(G)$ with $|S|=k$.
\end{definition}

Enomoto, Plummer and Saito \cite{Enomoto1999} obtained a relationship between neighborhoods of independent sets and $k$-factor-critical graphs. In \cite{Favaron 1996}, Favaron gave a sufficient condition for a graph to be $k$-factor-critical and proved that a $t$-connected graph $G$ with independence number $\alpha(G)$ at most $t-k+1$ is $k$-factor-critical.

\begin{lemma}\label{le1} {\normalfont(\cite{Favaron 1996})} Let $k$ and $t$ be nonnegative integers. Then every $t$-connected graph $G$ on $n$ vertices with $\alpha(G)\leq t-k+1$ and $n\equiv k~\rm{(mod~2)}$ is $k$-factor-critical.
\end{lemma}

Plummer and Saito \cite{Plummer2000} studied a relationship between $k$-factor-criticality and various closure operations and gave necessary and sufficient conditions for a graph to be $k$-factor-critical in terms of these closures.

Let $P$ be a property defined on all the graphs of order $n$, and let $l$ be a positive integer. We say that $P$ is $l$-stable if whenever $G+uv$ has the property $P$ with $d_{G}(u)+d_{G}(v)\geq l$, then $G$ itself has the property $P$. The concept of closure of a graph was used implicitly by Ore \cite{Ore1960},
and formally introduced by Bondy and Chvatal \cite{Bondy1976}.
Fix an integer $l\geq 0$, the $l$-closure of a graph $G$ is the graph obtained from $G$
by successively joining pairs of nonadjacent vertices whose degree sum is at least $l$
until no such pair exists. Denote by $\mathcal{C}_{l}(G)$ the $l$-closure of $G.$ Then we have
$$d_{\mathcal{C}_{l}(G)}(u)+d_{\mathcal{C}_{l}(G)}(v)\leq l-1$$ for any pair of nonadjacent vertices $u$ and $v$ of $\mathcal{C}_{l}(G).$

\begin{lemma}\label{le2} {\normalfont(\cite{Plummer2000})} Let $G$ be a connected graph of order $n$ and let $1\leq k\leq n-2$ be an integer. Then $G$ is $k$-factor-critical if and only if $\mathcal{C}_{n+k-1}(G)$ is $k$-factor-critical.
\end{lemma}

There are many sufficient conditions to assure that a graph contains an $[a, b]$-factor \cite{Fan2023, Fan2022, Kouider2004, Li2003, Li1998, Suil 2022}. O \cite{Suil 2021}, Wei and Zhang \cite{Wei2023} characterized sufficient conditions in terms of $e(G)$ to guarantee a graph to have $1$-factor and $k$-factor, respectively. In recent years, inspired by the related results of sufficient conditions of graph factor, the problem of finding spectral conditions for graphs having certain factor or structural properties has received considerable attention. Ao et al. \cite{Ao2023} investigated a sufficient condition in terms of the number of $r$-cliques to guarantee the existence of a $k$-factor in a graph with minimum degree at least $\delta$.

The sufficient conditions for the existence of $0$-factor-critical graphs have been studied in \cite{Liu2020, Suil 2021, Zhang 2021}. Recently, many researchers concentrated on the study of $k$-factor-critical graphs. Fan and Lin \cite{Fan Arx} provided a spectral condition for a connected graph with minimum degree $\delta$ to be $k$-factor-critical. Ao, Liu and Yuan \cite{Ao2025} improved the results of Fan and Lin \cite{Fan Arx}, and showed the number of $r$-cliques condition to guarantee a graph with minimum degree at least $\delta$ to be $k$-factor-critical. Zheng et al. \cite{Zheng 2024} established the following sufficient conditions in terms of the number of edges, signless Laplacian spectral radius $q(G)$ and distance signless Laplacian spectral radius $\lambda(G)$ to guarantee the existence of $k$-factor-criticality in a graph with minimum degree.

\begin{theorem}\label{th1} {\normalfont(\cite{Zheng 2024})} Suppose that $G$ is a connected graph of order $n\geq \max\{6\delta-5k+5, \frac{1}{6}[\delta^{2}-(2k-11)\delta+k^{2}-5k+10]\}$ with minimum degree $\delta\geq k$, where $n\equiv k~(\rm mod~2)$ and $k\geq1$. If $e(G)\geq e(K_{\delta}\vee (K_{n+k-2\delta-1}+(\delta-k+1)K_{1}))$, then $G$ is $k$-factor-critical unless $G\cong K_{\delta}\vee (K_{n+k-2\delta-1}+(\delta-k+1)K_{1})$.
\end{theorem}

\begin{theorem}\label{th2} {\normalfont(\cite{Zheng 2024})} Suppose that $G$ is a connected graph of order $n\geq \max\{\frac{19}{3}\delta-3k+3, \delta-1+(\delta+1)(\delta-k)\}$ with minimum degree $\delta\geq k$, where $n\equiv k~(\rm mod~2)$ and $k\geq1$. If $q(G)\geq q(K_{\delta}\vee (K_{n+k-2\delta-1}+(\delta-k+1)K_{1}))$, then $G$ is $k$-factor-critical unless $G\cong K_{\delta}\vee (K_{n+k-2\delta-1}+(\delta-k+1)K_{1})$.
\end{theorem}

\begin{theorem}\label{th3} {\normalfont(\cite{Zheng 2024})} Suppose that $G$ is a connected graph of order $n\geq \max\{12\delta-6k+5, \frac{2}{3}[\delta^{3}+\frac{4}{3}k\delta^{2}]\}$ with minimum degree $\delta\geq k$, where $n\equiv k~(\rm mod~2)$ and $k\geq1$. If $\lambda(G)\geq \lambda(K_{\delta}\vee (K_{n+k-2\delta-1}+(\delta-k+1)K_{1}))$, then $G$ is $k$-factor-critical unless $G\cong K_{\delta}\vee (K_{n+k-2\delta-1}+(\delta-k+1)K_{1})$.
\end{theorem}

For more results on $k$-factor-critical graphs, see \cite{Cai 2024, Zhai 2022, Zhang 2016, Zhang 2024, Zhou Si2024}. A spanning $k$-ended-tree of $G$ is a spanning tree with at most $k$ leaves, where $k\geq2$ is an integer. Motivated by the above and the results of spectral condition for spanning $k$-trees in $t$-connected graphs \cite{Li2024} or spanning $k$-ended trees in $t$-connected graph \cite{Zheng 2023}, we give a sufficient condition in terms of the number of edges for $k$-factor-criticality of $t$-connected graphs.

\begin{theorem}\label{th4}
Let $G$ be a $t$-connected graph of order $n\geq\frac{15t-11k+29}{2}$, where $t$ and $k$ are integers with $t\geq k\geq1$ and $n\equiv k~(\rm mod~2)$. If
$$e(G)>{n+k-t-2\choose 2}+(t-k+2)(t+1),$$
then $G$ is $k$-factor-critical unless $\mathcal{C}_{n+k-1}(G)\cong K_{t}\vee (K_{n+k-2t-1}+(t-k+1)K_{1})$.
\end{theorem}

Based on this edge condition, we obtain the following spectral radius condition.

\begin{theorem}\label{th5}
Let $G$ be a $t$-connected graph of order $n\geq\max\{\frac{15t-11k+29}{2}, t^{2}+\frac{5}{2}t+2\}$, where $t$ and $k$ are integers with $t\geq k\geq1$ and $n\equiv k~(\rm mod~2)$. If
$$\rho(G)\geq \rho(K_{t}\vee (K_{n+k-2t-1}+(t-k+1)K_{1})),$$
then $G$ is $k$-factor-critical unless $G\cong K_{t}\vee (K_{n+k-2t-1}+(t-k+1)K_{1})$.
\end{theorem}

In the next section, we introduce notations. In Sections \ref{se3} and \ref{se4}, we prove Theorems \ref{th4} and \ref{th5}, respectively.
\section{Preliminaries}
\label{sec:ch6-introduction}
Throughout this paper, we only consider simple and undirected graphs.
Let $G=(V(G), E(G))$ be a simple graph with vertex set $V(G)=\{v_{1}, v_{2},
\ldots, v_{n}\}$ and edge set $E(G)=\{e_{1},e_{2}, \ldots, e_{m}\}$. We use $|V(G)|=n$ and $e(G)$ to denote the number of vertices and edges of $G$, respectively. The degree $d_{G}(v)$ of a vertex $v$ in $G$ is the number of edges of $G$ incident with $v$. And the neighborhood $N_{G}(v)$ is the set of vertices of $G$ adjacent to $v$. For a subset $S \subseteq V(G)$ and a vertex $v \in V(G)$, $N_{S}(v)$ is the set of neighbors of $v$ in $S$, i.e., $N_S(v)=N_{G}(v)\cap S$ and $d_S(v)=|N_S(v)|$, that is, $d_S(v)$ is the number of edges between $v$ and the subset $S$ in $G$. For convenience, we use $\delta=\delta(G)$ to denote the minimum degree of $G$. A connected graph $G$ is called $t$-connected if it has more than $t$ vertices and remains connected whenever fewer than $t$ vertices are removed. For two disjoint vertex subsets $V_{1}$ and $V_{2}$ of $G$, we denote $e(V_{1}, V_{2})$ to be the number of edges each of which has one vertex in $V_{1}$ and the other vertex in $V_{2}$. Let $G[S]$ be the induced subgraph of $G$ whose vertex set is $S$ and whose edge set consists of all edges of $G$ which have both endpoints in $S$. A clique (resp., independent set) of $G$ is a subset $S$ of $V$ such that $G[S]$ is a complete (resp., empty) graph. The clique number (resp., independent number) of $G$, denoted by $\omega(G)$ (resp., $\alpha(G)$), is the number of vertices in a maximum clique (resp., independent set) of $G$. For two vertex-disjoint graphs $G$ and $H$, the disjoint union of $G$ and $H$, denoted by $G+H$, is the graph with vertex set $V(G)\cup V(H)$ and edge set $E(G)\cup E(H)$. In particular, let $tG$ be the disjoint union of $t$ copies of graph $G$. The join graph $G\vee H$ is obtained from $G+H$ by adding all possible edges between $V(G)$ and $V(H)$. We use $K_{n}$ to denote a complete graph of order $n$. For undefined terms and notions one can refer to \cite{Bondy2008, Brouwer2011, Godsil2001}.

The adjacency matrix $A(G)$ of a graph $G$ is an $n\times n$
matrix with $a_{ij}=1$ if $v_{i}$ and $v_{j}$ of $G$ are adjacent, and $a_{ij}=0$ otherwise. The largest eigenvalue of $A(G)$, denoted by $\rho(G)$, is called the spectral radius of $G$.

Since Danish mathematicians Petersen first attempted the study of factors in 1891, the graph factors theory plays an important role in graph theory. The definition of a $(g, f)$-factor of $G$ is a spanning subgraph $F$ of $G$ satisfying  $g(v)\leq d_{F}(v)\leq f(v)$ for any vertex $v$ in $V(G)$, where $g$ and $f$ are two integer-valued functions defined on $V(G)$ such that $0\leq g(v)\leq f(v)$ for each vertex $v$ in $V(G)$. A spanning subgraph of a graph is its subgraph whose vertex set is same as the original graph. Let $a$ and $b$ be two positive integers with $a\leq b$. A $(g, f)$-factor is called an $[a, b]$-factor if $g(v)\equiv a$ and $f(v)\equiv b$ for any $v\in V(G)$. In special, for a positive integer $k$, an $[a, b]$-factor is called a $k$-factor if $a=b=k$. When $k=1$, $1$-factor is also called a perfect matching, where a perfect matching is a set of nonadjacent edges covering every vertex of $G$.
\section{The edge condition}\label{se3}
Before presenting our main result, we first show that an $(n+k-1)$-closure of an non-$k$-factor-critical $t$-connected graph $G$ of order $n$ must contain a large clique if its number of edges is large enough. And we begin with the following Lemma \ref{le8}.

\begin{lemma}\label{le8}
Let $H$ be a $t$-connected and $(n+k-1)$-closure graph of order $n\geq\frac{15t-11k+29}{2}$, where $t$ and $k$ are integers with $t\geq k\geq1$ and $n\equiv k~(\rm mod~2)$. If
$$e(H)>{n+k-t-2\choose 2}+(t-k+2)(t+1),$$
then $\omega(H)\geq n+k-t-1$.
\end{lemma}
\begin{proof}
Since $H$ is a $t$-connected and $(n+k-1)$-closure graph of order $n\geq\frac{15t-11k+29}{2}$ (where $t$ and $k$ are integers with $t\geq k\geq 1$), any two vertices of degree at least $\frac{n+k-1}{2}$ must be adjacent in $H$. Let $C$ be the vertex set of a maximum clique in $H$, and let $\widetilde{C}$ be the vertex set in $H$ containing all vertices of degree at least $\frac{n+k-1}{2}$, that is, $\widetilde{C}=\{v\in V(H): d_{H}(v)\geq\frac{n+k-1}{2}\}$. Since $H$ is a $(n+k-1)$-closure, $\widetilde{C}$ is a clique in $H$. Let $H^{'}$ be the subgraph of $H$ induced by $V(H)\setminus C$ and $\widetilde{H}^{'}$ be the subgraph of $H$ induced by $V(H)\setminus \widetilde{C}$, where $|C|=r$, $|\widetilde{C}|=\widetilde{r}$, $|V(H^{'})|=n-r$ and $|V(\widetilde{H}^{'})|=n-\widetilde{r}$. Notice that $r\geq\widetilde{r}$.\

Note that $e(H[\widetilde{C}])={\widetilde{r}\choose 2}$, $e(V(\widetilde{H}^{'}), \widetilde{C})=\sum_{u\in V(\widetilde{H}^{'})}d_{\widetilde{C}}(u)$ and
$$e(\widetilde{H}^{'})=\frac{1}{2}\left(\sum_{u\in V(\widetilde{H}^{'})}d_{H}(u)-\sum_{u\in V(\widetilde{H}^{'})}d_{\widetilde{C}}(u)\right).$$
Next, we consider the following two cases on the order of $C$.

\vspace{1.5mm}
\noindent\textbf{Case 1.} $1\leq r\leq\frac{n-k+5}{3}+t$.
\vspace{1mm}\\
In this case, we have Claim \ref{cl1}.

\begin{claim}\label{cl1}
{\rm $d_{\widetilde{C}}(u)\leq \widetilde{r}$ and $d_{H}(u)\leq\frac{n+k-2}{2}$ for each $u\in V(\widetilde{H}^{'})$.}
\end{claim}

\begin{proof}
Due to $|\widetilde{C}|=\widetilde{r}$, we have $d_{\widetilde{C}}(u)\leq\widetilde{r}$ for each $u\in V(\widetilde{H}^{'})$. Note that
$$\widetilde{C}=\{v\in V(H): d_{H}(v)\geq\frac{n+k-1}{2}\}$$
and
$$V(\widetilde{H}^{'})=V(H)\setminus \widetilde{C}.$$
Thus, $d_{H}(u)\leq\frac{n+k-2}{2}$ for each $u\in V(\widetilde{H}^{'})$.
\end{proof}
By Claim \ref{cl1}, we obtain
\begin{align*}
e(H)=&e(H[\widetilde{C}])+e(\widetilde{H}^{'})+e(V(\widetilde{H}^{'}), \widetilde{C})\\
=&{\widetilde{r}\choose 2}+\frac{1}{2}\left(\sum_{u\in V(\widetilde{H}^{'})}d_{H}(u)-\sum_{u\in V(\widetilde{H}^{'})}d_{\widetilde{C}}(u)\right)+\sum_{u\in V(\widetilde{H}^{'})}d_{\widetilde{C}}(u)\\
=&{\widetilde{r}\choose 2}+\frac{1}{2}\left(\sum_{u\in V(\widetilde{H}^{'})}d_{H}(u)+\sum_{u\in V(\widetilde{H}^{'})}d_{\widetilde{C}}(u)\right)\\
\leq&{\widetilde{r}\choose 2}+\frac{1}{2}\left((n-\widetilde{r})\cdot\frac{n+k-2}{2}+(n-\widetilde{r})\cdot \widetilde{r}\right)\\
=&\frac{n-k}{4}\cdot \widetilde{r}+\frac{n^{2}+kn}{4}-\frac{n}{2}\\
\leq&\frac{n-k}{4}\cdot r+\frac{n^{2}+kn}{4}-\frac{n}{2}~~(\text {since}~\widetilde{r}\leq r)\\
\leq&\frac{n-k}{4}\cdot \left(\frac{n-k+5}{3}+t\right)+\frac{n^{2}+kn}{4}-\frac{n}{2}~~(\text {since}~r\leq\frac{n-k+5}{3}+t)\\
=&\frac{4n^{2}+(k+3t-1)n+k^{2}-5k-3kt}{12}\\
=&{n+k-t-2\choose 2}+(t-k+2)(t+1)-\frac{2n^{2}+(11k-15t-29)n}{12}\\
&-\frac{5k^{2}-21kt-37k+18t^{2}+66t+60}{12}\\
=&{n+k-t-2\choose 2}+(t-k+2)(t+1)-\frac{n(2n+11k-15t-29)}{12}\\
&-\frac{9k^{2}-21kt+\frac{49}{4}t^{2}+\frac{7}{4}t^{2}+4t^{2}-4k^{2}+66t-37k+60}{12}\\
=&{n+k-t-2\choose 2}+(t-k+2)(t+1)-\frac{n(2n+11k-15t-29)}{12}\\
&-\frac{(3k-\frac{7}{2}t)^{2}+\frac{7}{4}t^{2}+4(t^{2}-k^{2})+37(t-k)+29t+60}{12}~~(\text {since}~t\geq k\geq 1)\\
\leq&{n+k-t-2\choose 2}+(t-k+2)(t+1)-\frac{n(2n-15t+11k-29)}{12}\\
<&e(H)~~(\text {since}~n\geq\frac{15t-11k+29}{2}),
\end{align*}
which is a contradiction.

\vspace{1.5mm}
\noindent\textbf{Case 2.} $\frac{n-k+5}{3}+t<r\leq n+k-t-2$.
\vspace{1mm}
\begin{claim}\label{cl2}
{\rm For each $u\in V(H^{'})$, we have $d_{H}(u)\leq n+k-r-1$.}
\end{claim}
\begin{proof}
Assume that there exists a vertex $u\in V(H^{'})$ such that $d_{H}(u)\geq n+k-r$. Then $d_{H}(u)+d_{H}(v)\geq (n+k-r)+(r-1)=n+k-1$ for each $v\in C$. Note that $H$ is an $(n+k-1)$-closure graph. Then the vertex $u$ is adjacent to every vertex of $C$, hence, $C\cup\{u\}$ is a larger clique, which contradicts the maximality of $|C|$.
\end{proof}

By Claim \ref{cl2}, we obtain
\begin{align*}
e(H)=&e(H[C])+e(H^{'})+e(V(H^{'}), C)\\
=&{r\choose 2}+\frac{1}{2}\left(\sum_{u\in V(H^{'})}d_{H}(u)-\sum_{u\in V(H^{'})}d_{C}(u)\right)+\sum_{u\in V(H^{'})}d_{C}(u)\\
=&{r\choose 2}+\frac{1}{2}\left(\sum_{u\in V(H^{'})}d_{H}(u)+\sum_{u\in V(H^{'})}d_{C}(u)\right)\\
\leq&{r\choose 2}+\sum_{u\in V(H^{'})}d_{H}(u)\\
\leq&{r\choose 2}+(n-r)(n+k-r-1)\\
=&\frac{3}{2}r^{2}-\left(2n+k-\frac{1}{2}\right)r+n^{2}+kn-n.
\end{align*}
Let $f(r)=\frac{3}{2}r^{2}-\left(2n+k-\frac{1}{2}\right)r+n^{2}+kn-n$. Obviously, $f(r)$ is a convex function on $r$. Note that $\frac{n-k+5}{3}+t<r\leq n+k-t-2$.
\begin{align*}
e(H)\leq&\max\left\{f\left(\frac{n-k+5}{3}+t\right), f(n+k-t-2)\right\}\\
=&{n+k-t-2\choose 2}+(t-k+2)(t+1)
\end{align*}
for $n\geq\frac{15t-11k+29}{2}$, a contradiction.\

By Cases 1 and 2, we know that $\omega(H)=r\geq n+k-t-1$. The proof is completed.
\end{proof}

\medskip
\noindent \textbf{Proof of Theorem \ref{th4}.}
Suppose that $G$ is not a $k$-factor-critical graph, and $n\geq\frac{15t-11k+29}{2}$, $t\geq k\geq1$. Let $H=\mathcal{C}_{n+k-1}(G)$. It suffices to prove that $H\cong K_{t}\vee (K_{n+k-2t-1}+(t-k+1)K_{1})$.\

By Lemma \ref{le2}, $H$ is not $k$-factor-critical. As $G$ is $t$-connected, $H$ is also $t$-connected. Notice that $e(G)>{n+k-t-2\choose 2}+(t-k+2)(t+1)$ and $G\subseteq H$. Then
$$e(H)>{n+k-t-2\choose 2}+(t-k+2)(t+1).$$
By Lemma \ref{le8}, we have $\omega(H)\geq n+k-t-1$. Let $C$ be a maximum clique of $H$, and let $H^{'}$ be the subgraph of $H$ induced by $V(H)\setminus C$, where $|C|=r$ and $|V(H^{'})|=n-r$.\

Next, we characterize the structure of $H$. First we prove the following claims.

\begin{claim}\label{cl3}
{\rm $\omega(H)=|C|=n+k-t-1$.}
\end{claim}
\begin{proof}
Since $H$ is not a $k$-factor critical graph, by Lemma \ref{le1}, we have $\alpha(H)\geq t-k+2$. Hence, $\omega(H)\leq n+k-t-1$. By Lemma \ref{le8}, we have $\omega(H)\geq n+k-t-1$, we obtain that $\omega(H)=|C|=r=n+k-t-1$.
\end{proof}
According to Claim \ref{cl3}, the vertex set of $H$ can be partitioned as $V(H)=C\cup V(H^{'})$, where $C$ is a clique of size $n+k-t-1$. In the following, we need to consider the edge connection method between $C$ and $V(H^{'})$. Let $V(C)=\{u_{1}, u_{2}, \ldots, u_{n+k-t-1}\}$ and $V(H^{'})=\{v_{1}, v_{2}, \ldots, v_{t-k+1}\}$.
\begin{claim}\label{cl4}
{\rm $d_{H}(v_{i})=t$ for each $v_{i}\in V(H^{'})$.}
\end{claim}
\begin{proof}
Recall that $H$ is $t$-connected, then for every vertex $v_{i}\in V(H^{'})$, we always have $d_{H}(v_{i})\geq t$. Suppose that there exists a vertex $v_{i}\in V(H^{'})$ with $d_{H}(v_{i})\geq t+1$. Then
$$d_{H}(u_{i})+d_{H}(v_{i})\geq (n+k-t-2)+(t+1)=n+k-1$$
for each $u_{i}\in V(C)$. Note that $H$ is $(n+k-1)$-closure. Then the vertex $v_{i}$ is adjacent to every vertex of $C$. This implies that $C\cup\{v_{i}\}$ is a larger clique, a contradiction. Hence, $d_{H}(v_{i})=t$ for every $v_{i}\in V(H^{'})$.
\end{proof}
\begin{claim}\label{cl6}
{\rm $N_{H}(v_{i})\cap C=N_{H}(v_{j})\cap C$, where $i\neq j$.}
\end{claim}
\begin{proof}
For any $u_{i}\in N_{H}(v_{i})\cap C$ and $v_{j}\in V(H^{'})$, we have
$$d_{H}(u_{i})+d_{H}(v_{j})\geq (n+k-t-2)+1+t=n+k-1,$$
where $i\neq j$. Notice that $H$ is an $(n+k-1)$-closure graph. Then $u_{i}$ is adjacent to every vertex of $H^{'}$. It follows that $u_{i}\in N_{H}(v_{j})\cap C$, then $N_{H}(v_{i})\cap C\subseteq N_{H}(v_{j})\cap C$. Similarly, we can obtain that $N_{H}(v_{i})\cap C\supseteq N_{H}(v_{j})\cap C$. Then $N_{H}(v_{i})\cap C=N_{H}(v_{j})\cap C$, where $i\neq j$.
\end{proof}
\begin{claim}\label{cl5}
{\rm $V(H^{'})$ is an independent set.}
\end{claim}
\begin{proof}

Recall that $H$ is $t$-connected. By Claim \ref{cl4}, we get $1\leq|N_{H}(v_{i})\cap C|\leq t$ for every $v_{i}\in V(H^{'})$. By Claim \ref{cl3}, the vertex set of $H$ partitioned into two parts $C$ and $V(H^{'})$. In fact, if $1\leq|N_{H}(v_{i})\cap C|\leq t-1$, then $H[V(H)\setminus (N_{H}(v_{i})\cap C)]$ is disconnected, which contradicts the $t$-connectivity of $H$. Then $d_{C}(v_{i})=|N_{H}(v_{i})\cap C|=t$ for every $v_{i}\in V(H^{'})$. Combining with Claim \ref{cl3}, we know that $V(H^{'})$ is an independent set.
\end{proof}
By the above claims, we know that $H\cong K_{t}\vee (K_{n+k-2t-1}+(t-k+1)K_{1})$. Note that $K_{t}\vee (K_{n+k-2t-1}+(t-k+1)K_{1})$ is $t$-connected. By the definition of $k$-factor-critical graph, if we remove $k$ vertices of $K_{t}$~($t\geq k\geq 1$), then the vertices of $(t-k+1)K_{1}$ are only adjacent to the vertices of $K_{t-k}$. Thus, we cannot find a perfect matching in $K_{t-k}\vee (K_{n+k-2t-1}+(t-k+1)K_{1})$, and thus, the graph $G$ is not $k$-factor-critical.\\
And we have
\begin{align*}
e(K_{t}\vee (K_{n+k-2t-1}+(t-k+1)K_{1}))=&{n+k-t-1\choose 2}+t(t-k+1)\\
>&{n+k-t-2\choose 2}+(t-k+2)(t+1).
\end{align*}
This completes the proof.
\hspace*{\fill}$\Box$

\section{The spectral radius condition}\label{se4}
Let $A$ and $B$ be two $n\times n$ matrices. Define $A\leq B$ if $a_{ij}\leq b_{ij}$ for all $i$ and $j$, and define $A<B$ if $A\leq B$ and $A\neq B$.

\begin{lemma} {\normalfont(\cite{Berman1979, Horn1986})}\label{le9}
Let $A=(a_{ij})$ and $B=(b_{ij})$ be two $n\times n$ matrices with the spectral radius $\rho(A)$ and $\rho(B)$, respectively. If $0\leq A\leq B$, then $\rho(A)\leq\rho(B)$. Furthermore, if $0\leq A<B$, then $\rho(A)<\rho(B)$.
\end{lemma}

The following observation is very useful when we use the above upper bound on $\rho(G).$

\begin{lemma} {\normalfont(\cite{Hong1988})}\label{le4}
Let $G$ be a connected graph with $n$ vertices. Then
$$\rho(G)\leq \sqrt{2e(G)-n+1}.$$
\end{lemma}

\medskip
\noindent \textbf{Proof of Theorem \ref{th5}.}
Let $G$ be a $t$-connected graph of order $n\geq\max\{\frac{15t-11k+29}{2}, t^{2}+\frac{5}{2}t+2\}$, where $t\geq k\geq1$. Suppose to the contrary that $G$ is not $k$-factor-critical. Note that $K_{n+k-t-1}$ is a proper subgraph of $K_{t}\vee (K_{n+k-2t-1}+(t-k+1)K_{1})$. By Lemma \ref{le9}, we have
\begin{align}\label{in1}
\rho(G)\geq \rho(K_{t}\vee (K_{n+k-2t-1}+(t-k+1)K_{1}))>\rho(K_{n+k-t-1})=n+k-t-2.
\end{align}
Combining the above Inequality (\ref{in1}) and Lemma \ref{le4}, we obtain
\begin{align*}
\sqrt{2e(G)-n+1}\geq\rho(G)>n+k-t-2.
\end{align*}
Thus, we can get
\begin{align*}
e(G)>&\frac{(n-t+k-2)^{2}+n-1}{2}\\
=&{n+k-t-2\choose 2}+(t-k+2)(t+1)+n+\frac{3}{2}k+kt-t^{2}-\frac{7}{2}t-\frac{7}{2}\\
\geq&{n+k-t-2\choose 2}+(t-k+2)(t+1)+n+\frac{3}{2}+t-t^{2}-\frac{7}{2}t-\frac{7}{2}\\
\geq&{n+k-t-2\choose 2}+(t-k+2)(t+1)+n-t^{2}-\frac{5}{2}t-2\\
\geq&{n+k-t-2\choose 2}+(t-k+2)(t+1)
\end{align*}
for $n\geq\max\{\frac{15t-11k+29}{2}, t^{2}+\frac{5}{2}t+2\}$. Let $H\cong C_{n+k-1}(G)$. By Theorem \ref{th4}, we have $H\cong K_{t}\vee (K_{n+k-2t-1}+(t-k+1)K_{1})$. Note that $G\subseteq H$, so we have
$$\rho(G)\leq \rho(H)=\rho(K_{t}\vee (K_{n+k-2t-1}+(t-k+1)K_{1})).$$
Combining this with the condition $\rho(G)\geq \rho(K_{t}\vee (K_{n+k-2t-1}+(t-k+1)K_{1}))$, we have $G\cong H=K_{t}\vee (K_{n+k-2t-1}+(t-k+1)K_{1})$. Observe that $K_{t}\vee (K_{n+k-2t-1}+(t-k+1)K_{1})$ is not $k$-factor-critical. Thus the result follows.
\hspace*{\fill}$\Box$
\section*{Declaration of competing interest}
The authors declare that they have no conflict of interest.

\section*{Data availability}
No data was used for the research described in the article.

\section*{Acknowledgments}

The authors would like to thank the two anonymous referees for their helpful comments and valuable suggestions on improving the presentation of the paper.

\end{document}